\documentclass{amsart}
\usepackage{amsmath,amssymb,mathrsfs,hyperref,color}

\usepackage{subfigure}
\usepackage{float} 
\usepackage{times}
\usepackage{tikz}
\usetikzlibrary{intersections}
\usepackage{paralist}

\makeatletter
\def\setliststart#1{\setcounter{\@listctr}{#1}%
  \addtocounter{\@listctr}{-1}}
\makeatother

\makeatletter
\@addtoreset{figure}{section}
\makeatother

\usepackage{calc}
\newtheorem{The}{Theorem}[section]
\newtheorem{Cor}[The]{Corollary}
\newtheorem{Lem}[The]{Lemma}
\newtheorem{Pro}[The]{Proposition}

\newcounter{Mr}
\newtheorem{Result}[Mr]{\textbf{Main Result}}

\theoremstyle{definition}

\theoremstyle{remark}
\newtheorem{Rem}[The]{Remark}
\newtheorem{ex}[The]{Example}
\numberwithin{equation}{section}

\newcommand{\T}{\mathbb{T}}
\newcommand{\R}{\mathbb{R}}
\newcommand{\Z}{\mathbb{Z}}
\newcommand{\N}{\mathbb{N}}

\newcommand{\Q}{\mathbb{Q}}

\newcommand{\SING}{\mbox{\rm Sing}\, (u_c)}

\title[Dynamic and asymptotic behavior of singularities]{Dynamic and asymptotic behavior of singularities of certain weak KAM solutions on the torus}
\author{Piermarco Cannarsa \and Qinbo Chen \and Wei Cheng}
\address{Dipartimento di Matematica, Universit\`a di Roma ``Tor Vergata'',
Via della Ricerca Scientifica 1, 00133 Roma, Italy}
\email{cannarsa@mat.uniroma2.it}
\address{Morningside center of Mathematics, Academy of Mathematics and Systems Science, Chinese Academy of Sciences, Beijing 100190, China}
\email{chenqb@amss.ac.cn}
\address{Department of Mathematics, Nanjing University,
Nanjing 210093, China}
\email{chengwei@nju.edu.cn}
\date{\today}
\subjclass[2010]{35F21, 49L25, 37J50}
\keywords{Hamilton-Jacobi equation, weak KAM theory, generalized characteristic, singularities.}
\begin{document}
\begin{abstract}
	For mechanical Hamiltonian systems on the torus, we study the dynamical properties of the generalized characteristics semiflows associated with certain Hamilton-Jacobi equations, and build the relation between the $\omega$-limit set of this semiflow and the  projected Aubry set.
\end{abstract}
\maketitle

\section{Introduction}
The variational approach to Hamiltonian dynamical systems of Tonelli type have made  substantial progress in the past twenty years thanks to the celebrated contributions of Mather's theory, in Lagrangian formalism, and Fathi's weak KAM theory. In these theories, some specific invariant sets---namely, Mather sets, Aubry sets, and Ma\~n\'e sets---play the role of bridging the dynamic behavior of the given Hamiltonian system and the regularity properties of the  viscosity solution  to the associated Hamilton-Jacobi equation. The topological and metric properties of these invariant sets were widely studied in the past decades, as well as their dynamical implications. In this paper, we will study the onset, dynamic and asymptotic behavior of singularities for a certain class of weak KAM solutions on the torus.

Let $\T^n=\R^n/\Z^n$ be flat $n$-torus and $H=H(x,p):\T^n\times\R^n\to\R$ be a  Hamiltonian satisfying the standard Tonelli conditions: $H$ is of class $C^2$, and $p\mapsto H(x,p)$ is a convex function with uniformly superlinear growth. 

Let $u$ be a (locally semiconcave) viscosity solution of the Hamilton-Jacobi equation $H(x,Du(x))=0$. We say $x$ is a {\em critical point of $u$ with respect to $H$} if
\begin{equation}\label{def_crit}
	0\in\mbox{\rm co}\,H_p(x,D^+u(x)),
\end{equation}
where  $``\mbox{\rm co}"$  stands for convex hull and $D^+u(x)$ is the (Fr\'echet) superdifferential of $u$ at $x$. The set of all critical points of $u$ with respect to $H$ is denoted by 
\begin{equation}\label{critical_u_c}
\mbox{\rm Crit}_H\,(u)=\{x: 0\in\mbox{\rm co}\,H_p(x,D^+u(x))\}.
\end{equation} 
A point $x$ is called a {\em regular point of $u$ with respect to $H$} if $x\notin \text{\rm Crit}_H\,(u)$.

For any $c\in\R^n$, throughout this paper, we set
\begin{align*}
	H^c(x,p):=H(x,c+p)-\alpha_H(c),\quad (x,p)\in\T^n\times\R^n,
\end{align*}
where $\alpha_H(\cdot)$ is Mather's $\alpha$-function. From weak KAM theory,   there exists a $\T^n$-periodic weak KAM solution $u_c$ of the Hamilton-Jacobi equation
\begin{equation}\label{intr_HJ}
	H^c(x,Du_c(x))=0,\quad x\in \T^n.
\end{equation}

%
It is well known that $u_c$ may lose smoothness but is  semiconcave and Lipschitz continuous. We call $x\in\T^n$ a singular point of $u_c$ if $u_c$ fails to be differentiable at $x$, that is, the superdifferential $D^+u_c(x)$ is not a singleton. From now on, we denote by $\SING$ the set of all singular points of $u_c$. 

As noted in \cite{Albano-Cannarsa}, a key tool to describe the dynamics which governs the propagation of singularities of $u_c$ are the so-called generalized characteristics, i.e., solutions to the differential inclusion 
\begin{equation}\label{intr_gc}
\begin{cases}
\dot{\mathbf{x}}(s)\in\text{co}\,H^c_p(\mathbf{x}(s),D^+u_c(\mathbf{x}(s))), 
 &
 \mbox{a.e.}\,\, s\in[0,\tau]
 \vspace{.1cm}
 \\
 \mathbf{x}(0)=x_0.
 &
\end{cases}
\end{equation}
It is already known that if $x_0\in\SING$ and $x_0\in \text{Crit}_{H^c}\,(u_c)$, then \eqref{intr_gc} has a Lipschitz solution $\mathbf{x}$ such that $\mathbf{x}(s)\in\SING$ for all $s\in[0,\tau]$, with $\tau$ small enough, and $\mathbf{x}:[0,\tau]\to\T^n$ is injective (see \cite{Albano-Cannarsa}). We say a Hamiltonian $H$ {\em has the uniqueness property} if there exists a unique generalized characteristic staring from any given initial point.

In \cite{Cannarsa-Cheng3}, a global result for the propagation of singularities along  generalized characteristics was obtained. So, we are now in a position to start the analysis of the global behaviour of the  generalized characteristics semiflow. 

In this paper, we have decided to concentrate on the important example of mechanical systems on the torus $\T^n$. We have made this decision for several reasons. First, the uniqueness of the solution to \eqref{intr_gc} is not guaranteed for general Hamiltonians. So, the associated semiflow may fail to be well defined,  but such a semiflow is well defined for mechanical systems. Second, even for mechanical systems, the global theory of generalized characteristics on arbitrary manifolds is more complicated than the one on the torus. In general, this  semiflow depends on additional topological objects such as the critical set of the solution  of \eqref{intr_HJ}, that can be empty or nonempty as we show by examples at the end of Section 2. Such a complex structure, which is also related to Novikov's theory for the critical points of closed 1-forms (see \cite{Farber} and the references therein), will be the object of our future work. 

Let us consider a {\em mechanical Hamiltonian} of the following form:
\begin{equation}\label{eq:mech_intro}
	H(x,p)=\frac 12\langle A(x)p,p\rangle+V(x),\quad x\in\T^n, p\in\R^n
\end{equation}
where $x\mapsto A(x)$ is a $C^2$-smooth map taking values in the real space of $n\times n$ positive definite symmetric matrices, and $V$ is  a $C^2$-smooth potential function. Without loss of generality, we can assume  $\max\limits_{x} V(x)=0$. A typical Hamiltonian $H$ possessing the uniqueness property is the mechanical Hamiltonian \eqref{eq:mech_intro}. It is clear that any mechanical Hamiltonian $H$ in the form as \eqref{eq:mech_intro} can be looked as a Hamiltonian defined on $T^*\T^n$ or on $T^*\R^n$ and $\T^n$-periodic in the spatial variable $x$.

For the Hamiltonian  \eqref{eq:mech_intro}, it is worth mentioning that, by \eqref{def_crit}, for any viscosity solution $u_c$ of \eqref{intr_HJ}, $x$ is a critical point of $u_c$ with respect to $H^c$ if and only if
\begin{align*}
	0\in A(x)(c+D^+u_c(x)),
\end{align*}
namely
\begin{align}\label{cri_v}
	0\in c+D^+u_c(x)
\end{align}
since the matrix $A(x)$ is positive definite. 

By lifting to the universal covering space, $u_c$ can also be viewed as a function defined on $\R^n$,  in this case we define
\begin{equation}\label{intro:v}
v_c(x)=\langle c,x\rangle+u_c(x) \quad(x\in\R^n).
\end{equation} 
The inclusion formula \eqref{cri_v} leads us to define the following critical set, 
\begin{equation}\label{critical_v_c}
	\mbox{\rm Crit}\, (v_c)=\{x\in\R^n: 0\in D^+v_c(x)\}.
\end{equation} 
Any point $x\in \mbox{\rm Crit}\, (v_c)$ is called {\em a critical point of the function $v_c$}, otherwise we call it {\em a regular point}.
\begin{Rem}
It seems that we have provided two different definitions for critical point. However, to some extent, they are equivalent for the mechanical Hamiltonian \eqref{eq:mech_intro}. Indeed, $x$  is a critical point of $u_c$ with respect to $H^c$\ if and only if $x$ is a critical point of $v_c$. 
\end{Rem}

Therefore, in this paper we will discuss the dynamics of the generalized characteristics semiflows from two sides, one is the semiflow $\Phi_t$ on $\R^n$ (see \eqref{semiflow_noncmpt} for definition), the other is the semiflow $\phi_t$  on $\T^n$ ( see \eqref{semiflow_cmpt} for definition). 

The semiflow  $\Phi_t$  is essentially {\em gradient-like}. So, using a generalization of Conley's theory on gradient flows, we can show that

\begin{Result}
	Let $H$ be the mechanical Hamiltonian as in \eqref{eq:mech_intro} with $\Phi_t$ the associated semiflow of generalized characteristics. If
	 the regular values of $v_c$ are dense in $\R$,  then $\mathcal{R}(\Phi_t)\subset\mbox{\rm Crit}\, (v_c)$. In particular, we have
\begin{equation}
	L(\Phi_t):=\mbox{\rm cl}\, \Big(\cup\{\omega(\Phi_t,x): x\in X\}\Big)=\Omega(\Phi_t)=\mathcal{R}(\Phi_t)=\mbox{\rm Crit}\, (v_c),
\end{equation}
where $\Omega(\Phi_t)$ and $\mathcal{R}(\Phi_t)$ are the $\omega$-limit set and chain-recurrent set of $\Phi_t$, respectively.
\end{Result}

For the dynamics of the semiflow $\phi_t$  on $\T^n$, recall that Fathi introduced the set $\mathcal{I}_{u_c}$\footnote{We call $\mathcal{I}_u$ the projected Aubry set with respect to a weak KAM solution $u$. The precise definition of $\mathcal{I}_u$ is given in Section 3.} for a weak KAM solution $u_c$  of \eqref{intr_HJ} in \cite{Fathi-book}. Using the approach in \cite{Cannarsa-Cheng3}, we obtain some new results on the asymptotic behaviour of such semiflows.  These results build a bridge between the singular set $\SING$ and $\mathcal{I}_{u_c}$.

\begin{Result}
	Suppose $H$ is a mechanical Hamiltonian as in \eqref{eq:mech_intro} with $u_c$ a  weak KAM solution of \eqref{intr_HJ} and $x\in\SING$. Let $\mathbf{x}$ be the unique generalized characteristic starting from $x$ and $C(x)$ be the connected component of $\SING$  containing $x$. If $\mbox{\rm Crit}_{H^c} (u_c)\cap\overline{C(x)}=\emptyset$, then there exists a global generalized characteristic $\mathbf{y}:\R\to\T^n$ such that $\{\mathbf{y}(t): t\in\R\}$ is contained in $\overline{\{\mathbf{x}(t):t>0\}}\subset\overline{C(x)}$. Moreover, we have either
	\begin{enumerate}[\rm (1)]
	\item $\mathbf{y}:\R\to\T^n$ is a global singular generalized characteristic, or
	\item $\overline{C(x)}\cap\mathcal{I}_{u_c}\neq\emptyset$.  In particular, $\overline{C(x)}$ intersects the Aubry set.
	\end{enumerate}
\end{Result}

\begin{Result}
	Let $H$ be a mechanical Hamiltonian with $u_c$ a weak KAM solution of \eqref{intr_HJ}. Suppose that $x\in\SING$ and $\mathbf{x}$ is the unique generalized characteristic starting from $x$. If there exists a point of differentiability of $u_c$ in $\omega(\mathbf{x})$, the $\omega$-limit set of $\mathbf{x}$, then $\omega(\mathbf{x})$ intersects the Aubry set.
\end{Result} 

Main result 1 will be proved in section 2 and main result 2 and 3 will be shown in section 3. We also give an illustrative example of such connections between $\SING$ and the Aubry set in the context of area-preserving monotone twist maps.

The paper is organized as follows: in section 2, we mainly discuss the dynamics of generalized characteristics semiflows $\Phi_t$ defined on $\R^n$ and the relations among its various limiting sets. In section 3, we concentrate on the dynamics of generalized characteristics semiflows $\phi_t$ defined on $\T^n$, and show that the $\omega$-limit set of a singular generalized characteristic on $\T^n$ can intersect $\mathcal{I}_{u_c}$ under certain conditions. In section 4, we give an illustrative example of such a connection. In the appendix, we also give an explanation of how to transfer the results in \cite{Cannarsa-Cheng3} from $\R^n$ to $\T^n$.

\medskip

\noindent{\bf Acknowledgments} This work was partially supported by the Natural Scientific Foundation of China (Grant No. 11631006 and No.11790272) and the Istituto Nazionale di Alta Matematica ``Francesco Severi'' (GNAMPA 2017 Research Projects). The authors acknowledge the MIUR Excellence Department Project awarded to the Department of Mathematics, University of Rome Tor Vergata, CUP E83C18000100006.

\section{Generalized characteristics semiflows}
\subsection{Basic properties of generalized characteristics}

In this section, we will always refer to $H$ as in \eqref{eq:mech_intro}. It is well known that for any $c\in\R^n$ there is a unique $\alpha_H(c)\in \R$ such that  equation \eqref{intr_HJ} has a $\T^n$-periodic viscosity solution $u_c$. Let $v_c$ be as in \eqref{intro:v}. 

Notice that, By lifting to $\R^n$, the Hamiltonian  \eqref{eq:mech_intro} can also be regarded as a Hamiltonian $H(x,p):\R^n\times\R^n\to\R$ which is $\T^n$ periodic in $x$. The following proposition, based on the results in \cite{Cannarsa-Cheng}, \cite{Cannarsa-Cheng3}, \cite{Cannarsa-Yu} and \cite{Albano-Cannarsa}, is a collection of properties of the generalized characteristics associated with the pair $\{H,v_c\}$. 

\begin{Pro}\label{properties_g_c}
Let $H$ , $c\in\R^n$ and  $v_c$  be described as above. Fix any $x\in\R^n$.
\begin{enumerate}[\rm (a)]
  \item There is a unique Lipschitz arc $\mathbf{x}:[0,+\infty)\to\R^n$ such that
\begin{equation}
\label{eq:genchar}
\dot{\mathbf{x}}(s)\in A(\mathbf{x}(s))D^+v_c(\mathbf{x}(s))
\end{equation}
and $\mathbf{x}(0)=x$. Moreover, denoting by $\mathbf{y}$ the solution  of \eqref{eq:genchar} starting from the point $y\in\R^n$, we have 
	$$
	|\mathbf{x}(s)-\mathbf{y}(s)|\leqslant C|x-y|, \quad s\in[0,\tau]
	$$
for some constant $C\geqslant 0$ depending on $\tau>0$.
\item If $x\in\mbox{\rm Sing}\, (v_c)$, then $\mathbf{x}(s)\in\mbox{\rm Sing}\, (v_c)$ for all $s\in [0,+\infty)$.	
   \item The right derivative $\dot{\mathbf{x}}^+(s)$ exists for all $s\in[0,+\infty)$ and
\begin{equation*}
\label{generalized_characteristics_mech_sys}
\dot{\mathbf{x}}^+(s)=A(\mathbf{x}(s))p(s),\qquad\forall s\in[0,+\infty),
\end{equation*}
where $p(s)$ is the unique point of $D^+v_c(\mathbf{x}(s))$ such that
\begin{equation}\label{minimality}
\langle A(\mathbf{x}(s))p(s),p(s)\rangle=\min_{p\in D^+v_c(\mathbf{x}(s))}\langle A(\mathbf{x}(s))p,p\rangle.
\end{equation}
Moreover, $\dot{\mathbf{x}}^+(s)$ is right-continuous.
   \item The right derivative of $v_c(\mathbf{x}(\cdot))$ has the following representation:
   \begin{equation}
   	\frac d{ds^+}v_c(\mathbf{x}(s))=\langle p(s),A(\mathbf{x}(s))p(s)\rangle,\quad s\in[0,+\infty)
   \end{equation}
   where $p(s)$ is given in point (c) above.
   \item If  $\dot{\mathbf{x}}^+(s)\not=0$ for all $s\in [0,\tau]$, then 
\begin{equation}\label{eq:increasing}
v_c(\mathbf{x}(s_1))<v_c(\mathbf{x}(s_2)),\quad\ \forall~ 0\leqslant s_1<s_2\leqslant\tau.
\end{equation}
   \item If $0\notin D^+v_c(x)$, then $\mathbf{x}$ is injective on $[0,\tau]$ for some $\tau>0$.
   \item If $\alpha(c)>0$, then the set $\mbox{\rm Crit}\,(v_c)\subset\mbox{\rm Sing}\,(v_c)$.
 \end{enumerate}
\end{Pro}

Thus for any $x\in\R^n$, if we denote by $\mathbf{x}(\cdot,0,x)$ the unique generalized characteristic starting from $x$,  then it is clear that
\begin{equation}\label{semiflow_noncmpt}
\Phi_t(x):=\mathbf{x}(t,0,x),\quad (t,x)\in[0,\infty)\times\R^n	
\end{equation}
defines a semiflow on $\R^n$.

\subsection{Dynamic and topological structures of generalized characteristics semiflows}
Now we start to further analyze the properties of semiflows $\Phi_t$ on $\R^n$. Before that, we first introduce some basic facts on the the topological structure of semi-dynamical systems on a metric space $(X,d)$ (not necessarily compact). This can be considered as an extension of Conley's celebrating theory of general dynamical systems (see, for instance, \cite{Conley1978}, \cite{Conley1988} for the flows on compact metric space and \cite{Hurley} for the semiflows on arbitrary metric space).

\subsubsection{Conley's theory on semiflows}

A {\em semiflow} on $X$ is a continuous map $\Phi(t,x)=\Phi_t(x):[0,+\infty)\times X\to X$ such that $\Phi_0=\text{Id}$ where $\text{Id}$ stands for the identity map on $X$ and $\Phi_t\circ\Phi_s=\Phi_{t+s}$ for all $t\geqslant0$ and $s\geqslant 0$. Let $\mathscr{P}$ be the set of positive continuous functions on $X$. For any $\varepsilon\in\mathscr{P}$ and $T>0$, we define the {\em $(\varepsilon,T)$-chain} of a semiflow $\Phi_t$ as a finite number of pairs $(x_0,t_0),(x_1,t_1),\ldots,(x_n,t_n)\in X\times[0,+\infty)$ such that 
\begin{inparaenum}[(1)]
\item $t_j\geqslant T$ for $j=1,\ldots, n$;
\item $d(\Phi_{t_j}(x_j),x_{j+1})<\varepsilon(\Phi_{t_j}(x_j))$ for $j=0,\ldots, n-1$,
\end{inparaenum}
and the number $n$ is called the {\em length} of the chain. A point $x\in X$ is said to be {\em chain recurrent} if for any $\varepsilon\in\mathscr{P}$ and $T>0$ there exists an $(\varepsilon,T)$-chain with its length at least 1 connecting $x$ to itself. The set of all the chain recurrent points of a semiflow $\Phi_t$ is denoted by $\mathcal{R}(\Phi_t)$

Let $\Phi_t$ be a semiflow on $X$, a nonempty open set $U\subset X$ is said to be a {\em preattractor} of $\Phi_t$ if there exists $T>0$ such that
$$
\overline{\Phi([T,+\infty)\times U)}\subset U.
$$
An {\em attractor} $A$ defined by a preattractor $U$ is determined by
$$
A=\bigcap_{t\geqslant0}\overline{\Phi([t,+\infty)\times U)}=\bigcap_{t\geqslant T}\overline{\Phi([t,+\infty)\times U)}.
$$
The {\em basin} of an attractor $A$, denoted by $\mathcal{B}(A)$, consists of all point $x$ with the property that some point on the forward orbit of $x$ lies in a preattractor $U$ that determines $A$. 

Now, we can formulate an extension of Conley's theorem.

\begin{Pro}[\cite{Hurley}]\label{Conley_Thm}
Suppose $(X,d)$ is a metric space and $\Phi_t$ is a semiflow on $X$. Then we have
$$
X\setminus\mathcal{R}(\Phi_t)=\bigcup\limits_{A}~(\mathcal{B}(A)\setminus A),
$$	
where $A$ varies over the collection of attractors of $\Phi_t$ and $\mathcal{B}(A)$ denotes the basin of attraction of $A$.
\end{Pro}

\subsubsection{Limiting sets of generalized characteristics semiflows}

Consider the semiflows $\Phi_t$ of generalized characteristics \eqref{semiflow_noncmpt} defined on $X=\R^n$. For any $r\in\R$, it is standard to define $U_r=\{x: v_c(x)>r\}$. 
\begin{Lem}\label{preattractor}
Let $r\in\R$ be a regular value of $v_c$, then $\cup_{t>0}\overline{\Phi_t(U_r)}\subset U_r$. In other words, $U_r$ is a preattractor.
\end{Lem}

\begin{proof}
We prove by contradiction. We fix any $t>0$. Suppose $\overline{\Phi_t(U_r)}\setminus U_r\not=\varnothing$, then there exists a sequence $\{x_k\}$ in $U_r$ with $y_k=\Phi_t(x_k)$ such that $\lim\limits_{k\to\infty}y_k=z\in\overline{\Phi_t(U_r)}\setminus U_r$. Thus, we must have $v_c(z)=r$. Since $r<v_c(x_k)<v_c(y_k)\to v_c(z)=r$, we have
\begin{equation}\label{eq:u_x_n_y_n}
	\lim_{k\to\infty}(v_c(y_k)-v_c(x_k))=0.
\end{equation}
Now, by Proposition \ref{properties_g_c}, we denote by $\mathbf{x}_k:[0,+\infty)\to\R^n$ the global generalized characteristic starting from $x_k$ for each $k\in\N$, and each $\mathbf{x}_k$ has the property that $\dot{\mathbf{x}}_k^+(s)$ exists for all $s\in[0,+\infty)$ and $\dot{\mathbf{x}}_k^+(s)=A(\mathbf{x}_k(s))p_k(s)$ is right-continuous, where $p_k(s)$ is the unique point of $D^+v_c(\mathbf{x}_k(s))$ satisfying \eqref{minimality}. We also have
\begin{align*}
\frac d{ds^+}v_c(\mathbf{x}_k(s))=\langle p_k(s),A(\mathbf{x}(s))p_k(s)\rangle.
\end{align*}
Therefore, by H\"older's inequality and the equality above, we have
\begin{align*}
	v_c(y_k)-v_c(x_k)=&v_c(\mathbf{x}_k(t))-v_c(\mathbf{x}_k(0))=\int^t_0\langle p_k(s),A(\mathbf{x}_k(s))p_k(s)\rangle\ ds\\
	=&\int^t_0\langle A^{-1}(\mathbf{x}_k(s))\dot{\mathbf{x}}_k^+(s),\dot{\mathbf{x}}_k^+(s)\rangle\ ds\geqslant\lambda\int^t_0|\dot{\mathbf{x}}_k^+(s)|^2\ ds\\
	\geqslant&\frac{\lambda}t\Big(\int^t_0|\dot{\mathbf{x}}_k^+(s)|\ ds\Big)^2\geqslant\frac{\lambda}t|\mathbf{x}_k(t)-\mathbf{x}_k(0)|^2=\frac{\lambda}t|y_k-x_k|^2.
\end{align*}
It follows that $\lim_{k\to\infty}x_k=z$ by \eqref{eq:u_x_n_y_n} and $\Phi_t(z)=z$. Since $r$ is a regular value of $v_c$, then  $r=v_c(z)<v_c(\Phi_t(z))$. This leads to a contradiction.
\end{proof}

For a continuous map $f: X\to X$ of the topological space $X$, a point $x\in X$ is called {\em non-wandering} if for any open set $U$ containing $x$, there is an integer $k\geq 1$ such that $f^k(U)\cap U\neq\emptyset$. The set of all non-wandering points of $f$ is denoted by $NW(f)$. We also need the standard concept of topological entropy for (semi-)dynamical systems (see. for instance, \cite{Katok})

\begin{Cor}
The time one map of the generalized characteristics semiflow $\Phi_t(\cdot)$ has zero topological entropy, i.e., $h_{top}(\Phi_1)=0$.
\end{Cor}
\begin{proof}
To calculate the topological entropy $h_{top}(\Phi_1)$, we could use the following equality (see, for example \cite{Katok}): 
$$
h_{top}(\Phi_1)=h_{top}\big( \Phi_1\big|_{NW(\Phi_1)} \big).
$$
By Lemma \ref{preattractor}, one could easily verify that if $x$ is a regular point of $v_c$, then $x\notin NW(\Phi_1)$. On the other hand, if $x$ is critical point of $v_c$, $\mathbf{x}(t)\equiv x$ is the unique generalized characteristic starting from $x$, which implies that $x\in NW(\Phi_1)$. Therefore, $NW(\Phi_1)=\{x : 0\in D^+v_c(x)\}$.

Since $\Phi_1\big|_{NW(\Phi_1)}$ is exactly the identity map, we can obviously obtain that  $h_{top}(\Phi_1)=h_{top}\big( \Phi_1\big|_{NW(\Phi_1)} \big)=0$.
\end{proof}

\begin{The}\label{critical_set}
Let the following regularity condition be satisfied:\\
{\rm (R)} \ the regular values of $v_c$ are dense in $\R$.\\
Then $\mathcal{R}(\Phi_t)\subset\mbox{\rm Crit}\, (v_c)$. In particular, we have
\begin{equation}\label{eq:relation_limit_sets}
	L(\Phi_t):=\mbox{\rm cl}\, \Big(\cup\{\omega(\Phi_t,x): x\in X\}\Big)=\Omega(\Phi_t)=\mathcal{R}(\Phi_t)=\mbox{\rm Crit}\, (v_c),
\end{equation}
where $\Omega(\Phi_t)$ (resp. $\omega(\Phi_t,x)$) is the $\omega$-limit set of the semiflow $\Phi_t$ (resp. the semi-orbit $\Phi_t(x)$), and $\mathcal{R}(\Phi_t)$ is the chain-recurrent set of $\Phi_t$.
\end{The}

\begin{Rem}
	It is possible that the above various invariant sets can be empty simultaneously (See the  examples below).
\end{Rem}

\begin{proof}
	First, it is clear by Proposition \ref{Conley_Thm} that $x\in X$ is not a chain recurrent point of $\Phi$ if and only if there exists a preattractor $U$ which determines a attractor $A$ such that there exists $T\geqslant 0$, $\Phi_{T}(x)\in U$ but $x\not\in A$.
	
	Now let us prove $\mathcal{R}(\Phi_t)\subset\{x : 0\in D^+v_c(x)\}$.  If $0\notin D^+v_c(x)$, then $x$ is a regular point of $v_c$, hence for any $t>0$, $v_c(\Phi_t(x))>v_c(x)$. Since condition (R), there exists a regular value  $r=r(t)$ such that $v_c(\Phi_t(x))>r>v_c(x)$, then $U_r$ is a preattracttor by Lemma \ref{preattractor} and $\Phi_t(x)\in U_r$, but $x\not\in U_r$. This implies that $x$ is not a chain recurrent point of $\Phi_t$. 
	
	To prove \eqref{eq:relation_limit_sets}, we observe the following relations of relevant invariant sets is standard:
	$$
	L(\Phi_t)\subset\Omega(\Phi_t)\subset\mathcal{R}(\Phi_t).
	$$
	Moreover, if $0\in D^+ v_c(x)$, $\mathbf{x}(t)\equiv x$ is the unique generalized characteristic starting from $x$, then $x\in L(\Phi_t)$.
\end{proof}

\begin{ex}
	Let the Hamiltonian $H$ in \eqref{eq:mech_intro} be a nearly-integrable systems, i.e. 
	\begin{equation}\label{eq:nearly_integrable}
		H(x,p)=\frac 12\langle A(x)p,p\rangle+\varepsilon V(x),\quad x\in\T^n, p\in\R^n,
	\end{equation}
	where $\varepsilon>0$ is small. Then, it is not hard to be checked that for any fixed $0\not=c\in\R^n$, there exists $\varepsilon_0=\varepsilon_0(c)>0$ such that the Lipschitz constants of the visocisty solutions $u^{\varepsilon}_c$ are bounded by $D\sqrt{\varepsilon}$ for any $\varepsilon\in(0,\varepsilon_0)$, where $D=D(c)>0$ is independent of $c$ (see, e.g., \cite[Theorem 1.1]{Chen-Zhou} or \cite{Bernard2000}).  This implies that there is no critical point of $v_c$ if $0<\varepsilon<\kappa(c)=\min\{\varepsilon_0(c),\frac{|c|^2}{D(c)^2}\}$. Therefore, we have the following result:
\end{ex}

\begin{Cor}\label{no_critical}
Let $H$ be in the form as \eqref{eq:nearly_integrable} and $c\not=0$, then there exists a constant $\kappa(c)>0$ such that the set $\mbox{\rm Crit}\, (v_c)=\varnothing$ if $0<\varepsilon<\kappa(c)$.
\end{Cor}

Condition (R) in Proposition \ref{critical_set} is essential. It is closely connected to the Morse-Sard's type results, see, for instance, \cite{Rifford} for the viscosity solutions and \cite{Ferry} for the distance functions, also \cite{Fathi-Figalli-Rifford}, \cite{Bernard2010}, \cite{Sorrentino2008}, \cite{Mather2003} and \cite{Mather2004} for the related results on the quotient Aubry sets.

\begin{Pro}[\cite{Rifford}]\label{Rifford}
Let $H$ be a Hamiltonian in the form as \eqref{eq:mech_intro}, then Condition {\rm (R)} in Proposition \ref{critical_set} is satisfied if any of the following two conditions is satisfied
\begin{enumerate}[\rm (1)]
  \item $n=2$.
  \item $n=3$, $H$ is of class $C^4$ and $c=0$;
  \item $n=3$, $H$ is real-analytic.
\end{enumerate}
\end{Pro}

However, for $n=3$, the smoothness requirements of $H$ in Proposition \ref{Rifford} are restrictive. In fact, \cite{Rifford} has proven that if $n=3$, the condition {\rm (R)} holds for generic $C^2$ Hamiltonians: let the Hamiltonian $H$ in \eqref{eq:mech_intro} be $C^2$ and $c\in\R^n$, then there exists an open dense set $\mathcal{O}\subset C^2(\T^n,\R)$, such that for every function $P\in\mathcal{O}$, the corresponding $v_c$ of $H+P$ satisfies the condition {\rm (R)}. 

\begin{ex}
	As shown in Corollary \ref{no_critical}, the set $\mbox{\rm Crit}\, (v_c)$ can be empty for  nearly-integrable systems. For another example, the mathematical pendulum $H(x,p)=\frac{1}{2}p^2+\cos 2\pi x-1$, we have: If $c\in[-\frac{4}{\pi},\frac{4}{\pi}]$, then $\alpha(c)=0$ and the set $\mbox{\rm Crit}\, (v_c)$ must be nonempty; If $c\notin[-\frac{4}{\pi},\frac{4}{\pi}]$, then  $\alpha(c)>0$ and the set $\mbox{\rm Crit}\, (v_c)$ must be empty. 
\end{ex}


It is still unclear what is the essential condition to ensure the existence or nonexistence of critical points for $v_c$. However we have the following results.

\begin{The}\label{th:critical1}
If $S$ is a bounded connected component of $\mbox{\rm Sing}\, (v_c)$, then $S\cap \mbox{\rm Crit}\, (v_c)\neq\emptyset$.
\end{The}
\begin{proof}
Fix  $x_0\in S\cap\mbox{\rm Sing}\, (v_c)$ and let $\mathbf{x}:[0,+\infty)\to\R^n$ be the generalized singular characteristic with initial point $x_0$. Then $\mathbf{x}(t)\in S$ for all $t\in [0,+\infty)$ and, by Proposition \ref{properties_g_c},
\begin{equation}
\label{eq:increase_of_v}
v_c(\mathbf{x}(t))-v_c(x_0)=\int_0^t\langle p(s),A(\mathbf{x}(s))p(s)\rangle\,ds\quad\forall t\geqslant 0
\end{equation}
where $p(\cdot)$ satisfies \eqref{minimality}.

If, for some $t\geqslant 0$, $\mathbf{x}(t)$ is  critical for $v_c$, then the conclusion is proved. 

Suppose next that $\mathbf{x}(t)$ is not critical for $v_c$ for all $t\geqslant 0$. Then we claim that
\begin{equation}
\label{eq:abs-abs}
\lim_{j\to\infty}\langle p(s_j),A(\mathbf{x}(s_j))p(s_j)\rangle=0
\end{equation}
for some sequence $s_j\geqslant 0$. 
For if  
\begin{equation*}
\delta:=\inf_{t\geqslant 0} \langle p(s),A(\mathbf{x}(s))p(s)\rangle>0\,,
\end{equation*}
then, appealing to \eqref{eq:increase_of_v}, we obtain
 \begin{equation*}
v_c(\mathbf{x}(t))-v_c(x_0)\geqslant \delta t \quad\forall t\geqslant 0\,,
\end{equation*}
in contrast with the fact that $v_c$ is bounded on $S$. So, \eqref{eq:abs-abs} is proved.

Now, since $S$ is compact and the set-valued map $x\rightrightarrows D^+v_c(x)$ is upper semicontinuous, 
we can also assume that $\mathbf{x}(s_j)\to \bar x\in S$  and
 $p(s_j)\to \bar p\in D^+v_c(\bar x)$ as $j\to\infty$. Thus, 
 $\langle \bar p,A(\bar x)\bar p\rangle=0$ by \eqref{eq:abs-abs}.
 Since  $A(\bar x)$ is positive definite, this implies that $\bar p=0$ and $\bar x$ is a critical point of $v_c$. 
\end{proof}

%
%

\section{Singularities may approach Aubry sets}

In  Section 2, we have already shown that, if a certain Morse-Sard type condition (R) is satisfied,  the generalized characteristics semiflows $\Phi_t$ on $\R^n$ is essentially a gradient-like flow. But since we lift the solutions to the universal covering space $\R^n$, which is not compact, we cannot obtain the asymptotic results by just studying the $\omega$-limit set of a given generalized characteristic.  So we have to return to the torus and begin to study the asymptotic behavior of the generalized characteristics semiflow on $\T^n$.


In this section, we still suppose that $H(x,p):\T^n\times\R^n\to\R$ has the form as in \eqref{eq:mech_intro}. Recall that for any $c\in\R^n$, $H^c=H(x,c+p)-\alpha_H(c)$ and the associated Lagrangian 
$$
L^c(x,v)=\frac{1}{2}\langle A^{-1}(x)v,v\rangle -V(x)-\langle c,v\rangle+\alpha_H(c).
$$
It is clear that the Ma\~n\'e's critical value with respect to $H^c$ is $0$. 

Similar to the last section, there is a semiflow $\phi_t$ defined on the torus $\T^n$ such that
\begin{equation}\label{semiflow_cmpt}
	\phi_t(x):[0,+\infty)\times\T^n\to\T^n \quad\text{and}\quad \phi_t\circ \pi=\pi\circ\Phi_t,
\end{equation}
where $\pi:\R^n\to\T^n$ is the canonical projection and $\Phi_t$ is defined in \eqref{semiflow_noncmpt}. By the the periodicity of $H$, 
it is easy to find that for any $x\in\T^n$, $\phi_t(x)$ is the unique generalized characteristic of the  differential inclusion
\begin{equation}\label{eq:genchar1}
\begin{cases}
\dot{\mathbf{x}}(s)\in A\big(\mathbf{x}(s)\big)\big(c+D^+u_c(\mathbf{x}(s))\big), 
 &
 \mbox{a.e.}\,\, s
 \vspace{.1cm}
 \\
 \mathbf{x}(0)=x.
 &
\end{cases}
\end{equation}

Therefore, it follows from Proposition \ref{properties_g_c} and the formula \eqref{intro:v} that
\begin{Pro}\label{properties_g}
Let $H$ , $c\in\R^n$ and  $u_c$  be described as above.  Fix any $x\in\T^n$.
\begin{enumerate}[\rm (a)]
  \item There is a unique Lipschitz arc $\mathbf{x}:[0,+\infty)\to\T^n$ satisfying \eqref{eq:genchar1}. Indeed, this arc $$\mathbf{x}(t)=\phi_t(x), \quad\text{for all~} t\geqslant 0.$$  If $\mathbf{x}, \mathbf{y}$ are two solutions  of \eqref{eq:genchar1} starting from  $x, y\in\T^n$ respectively, then
	$$
	|\mathbf{x}(s)-\mathbf{y}(s)|\leqslant C|x-y|, \quad s\in[0,\tau]
	$$
for some positive constant $C$ depending on $\tau>0$.
\item If $x\in\mbox{\rm Sing}\, (u_c)$, then $\mathbf{x}(s)\in\mbox{\rm Sing}\, (u_c)$ for all $s\in [0,+\infty)$.	
   \item The right derivative $\dot{\mathbf{x}}^+(s)$ exists for all $s\in[0,+\infty)$ and
\begin{equation*}
\dot{\mathbf{x}}^+(s)=A\big(\mathbf{x}(s)\big)\big(c+p(s)\big),\qquad\forall s\in[0,+\infty),
\end{equation*}
where $p(s)$ is the unique point of $D^+u_c(\mathbf{x}(s))$ such that
\begin{equation}
\langle A\big(\mathbf{x}(s)\big)\big(c+p(s)\big),\big(c+p(s)\big)\rangle=\min_{p\in D^+u_c(\mathbf{x}(s))}\langle A\big(\mathbf{x}(s)\big)\big(c+p\big),\big(c+p\big)\rangle.
\end{equation}
Moreover, $\dot{\mathbf{x}}^+(s)$ is right-continuous.
   \item If $x\notin \mbox{\rm Crit}_{H^c}(u_c)$, then $\mathbf{x}$ is injective on $[0,\tau]$ for some $\tau>0$.
   \item If $\alpha(c)>0$, then the set  $\mbox{\rm Crit}_{H^c}(u_c)$ $\subset$ $\mbox{\rm Sing}\,(u_c)$. 
 \end{enumerate}
\end{Pro}

Let us recall some basic facts from weak KAM theory (see, for instance, \cite{Fathi-book}). A  function $u:\T^n\to\R$ is said to be {\em dominated} by $L^c$ iff, for each absolutely continuous arc $\gamma:[a,b]\to\T^n$ with $a<b$, one has
$$
u(\gamma(b))-u(\gamma(a))\leqslant\int^b_aL^c(\gamma(s),\dot{\gamma}(s))ds.
$$
When this happens, one writes $u\prec L^c$. An absolutely continuous curve $\gamma:[a,b]\to\T^n$ is said to be $(u,L^c,0)$-{\em calibrated} if
$$
u(\gamma(b))-u(\gamma(a))=\int^b_aL^c(\gamma(s),\dot{\gamma}(s))ds.
$$
For any weak KAM solution $u_c$ of \eqref{intr_HJ}, we define 
\begin{align*}
	\mathcal{I}_{u_c}:=\{x\in\T^n: x=\gamma(0)\ \text{for some}\ (u_c,L^c,0)\text{-calibrated curve}\ \gamma:\R\to\T^n\},
\end{align*}
and
\begin{align*}
	\mathcal{A}(H^c):=\bigcap\limits_{u_c}\mathcal{I}_{u_c}.
\end{align*}
where $u_c$ is taken over all weak KAM solutions $u_c$ of \eqref{intr_HJ}. Such set $\mathcal{A}(H^c)$ is called the {\em projected Aubry set} with respect to the Hamiltonian $H^c$.


For any $t>0$, given $x, y\in\T^n$, we set
$$
\Gamma^t_{x,y}=\{{\xi\in W^{1,1}([0,t];\T^n): \xi(0)=x,\xi(t)=y}\}
$$
and define
\begin{equation}\label{fundamental_solution}
A_{t}(x,y)=\min_{\xi\in\Gamma^t_{x,y}}\int^{t}_{0}L^c(\xi(s),\dot{\xi}(s))ds.
\end{equation}
The existence of the above minimum is a well-known result in Tonelli's theory (see, for instance, \cite{Fathi-book}). Any arc $\xi\in\Gamma^t_{x,y}$ at which the minimum in \eqref{fundamental_solution} is achieved will be called a {\em minimizer} for $A_t(x,y)$. Any minimizer $\xi$ is of class $C^2$ by classical results. 

For further analysis, we  need more information about the construction of the generalized characteristics which was first discovered in \cite{Cannarsa-Cheng3} (see, also \cite{CCF}).
For any $t>0$ and any continuous function $u:\T^n\to(-\infty,+\infty)$, the {\em Lax-Oleinik operators} $T^{\pm}_tu:\T^n\to(-\infty,+\infty)$ are defined as follows:
\begin{gather}
T^+_tu(x)=\sup_{y\in\T^n}\{u(y)-A_t(x,y)\},\quad x\in\T^n,\label{L-L regularity_sup}\\
T^-_tu(x)=\inf_{y\in\T^n}\{u(y)+A_t(y,x)\},\quad x\in\T^n.\label{L-L regularity_inf}
\end{gather}

Now, let $u_c$ be a weak KAM solution of \eqref{intr_HJ} and let $A_t(x,y)$ be the fundamental solution with respect to $L^c$. By applying Lemma 3.2 in \cite{Cannarsa-Cheng3} to \eqref{intr_HJ} (see Appendix A for more details), we conclude that there exist  $t_0>0$ and $\lambda_0>0$ such that, for any $x\in\T^n$ and $t\in(0,t_0]$, the function $u_c(\cdot)-A_t(x,\cdot)$ admits a unique maximizer $y_{t,x}$, and $y_{t,x}\in B(x,\lambda_0 t)$. In addition, this function $A_t(x,\cdot)$ is differentiable at any $y\in B(x,\lambda_0 t)$. Let
\begin{equation}\label{eq:bold_y}
\mathbf{y}_x(t):=\begin{cases}
x&\mbox{if}\quad t=0,\\
y_{t,x}&\mbox{if}\quad t\in(0,t_0].
\end{cases}	
\end{equation}
Then $\mathbf{y}_x(t)\in\SING$ if $x\in\SING$. The generalized characteristics semiflow $\phi_t$ in Proposition \ref{properties_g} is  indeed constructed by
\begin{align*}
	\phi_t(x)=\mathbf{y}_x(t),  \quad t\in[0,t_0],\ x\in\T^n.
\end{align*}
and determined inductively as
\begin{equation}\label{eq:rescaled_GC}
	\phi_t(x)=\mathbf{y}_{\phi_{it_0}(x)}(t-it_0), \quad\forall t\in[it_0,(i+1)t_0],\,i\in\N.
\end{equation}  

\begin{Rem}\label{using_tau}
	Owing to the uniqueness of generalized characteristics when $H$ has the form \eqref{eq:mech_intro}, for any $0<\tau\leqslant t_0$ we also have
	\begin{equation}\label{eq:using_tau}
		\phi_t(x)=\mathbf{y}_{\phi_{i\tau}(x)}(t-it_0),\quad\forall t\in[i\tau,(i+1)\tau],\,i\in\N.
	\end{equation}
\end{Rem}

\begin{Lem}\label{limiting_GC}
	Fix $t_0>0$. If $\{x_k\}$ is a sequence which converges to $x$ as $k\to\infty$ and $\mathbf{x}$, $\mathbf{x}^{k}$ are the unique generalized characteristic defined by \eqref{eq:rescaled_GC} with initial conditions $\mathbf{x}(0)=x$ and $\mathbf{x}^k(0)=x_k$ for all $k\in\N$, respectively, then $\mathbf{x}^{k}\vert_{[0,it_0]}$ converges to $\mathbf{x}\vert_{[0,it_0]}$ uniformly for any $i\in\N$ as $k\to\infty$.
\end{Lem}

\begin{proof}
	Let $\mathbf{y}_k:[0,t_0]\to\T^n$ be a sequence of generalized characteristics defined as in \eqref{eq:bold_y} with initial points $x_k$, i.e.,
	$$
	T^+_su_c(x_k)=u_c(\mathbf{y}_k(s))-A_s(x_k,\mathbf{y}_k(s)),\quad 0<s\leqslant t_0.
	$$
	By Proposition \ref{properties_g} (a), we have that $\mathbf{y}_k$ converges to $\mathbf{y}$ uniformly on $[0,t_0]$. Therefore
	$$
	T^+_su_c(x)=u_c(\mathbf{y}(s))-A_s(x,\mathbf{y}(s)),\quad 0<s\leqslant t_0,
	$$
	by the continuity of the function $F(z,y)=T^+_su_c(z)-u_c(y)+A_s(z,y)$. It follows that $\mathbf{y}(s)$ is the maximizer of $u_c(\cdot)-A_s(x,\cdot)$. Thus, $\mathbf{y}:[0,t_0]\to\T^n$ is Lipschitz continuous by Proposition \ref{properties_g}  and $\mathbf{y}$ coincides with $\mathbf{x}$ on $[0,t_0]$. The rest of the proof is obtained by induction. 
\end{proof}


For any $\tau\in(0,t_0]$ where $t_0$ is fixed and described as above, let $z^{\tau}_i=\mathbf{x}(i\tau)$, $i\in\N$, and let $\mathcal{Z}^{\tau}$ be the set of all convergent subsequences of $\{z^{\tau}_i\}$. For any strictly increasing sequence of natural numbers $\sigma=\{i_1,i_2,\ldots,i_k,\ldots\}$ and the associated convergent subsequence $z^{\tau}_{\sigma}=\{z^{\tau}_{i_k}\}$, we define
	$$
	N^{\tau}_{\sigma}=\sup\{i_{k+1}-i_k: z^{\tau}_{\sigma}\in\mathcal{Z}^{\tau}\}.
	$$
	If $\tau=t_0$, we will take out the superscript $\tau$ for abbreviation. It is clear $\mathbf{x}$ is unique, then it does not depend on the choice of $\tau\in(0,t_0]$ by \eqref{eq:using_tau}.
	
We denote the $\omega$-limit set of $\mathbf{x}$ by $\omega(\mathbf{x})$. It is clear that $\omega(\mathbf{x})\subset\overline{C(x)}$ where $C(x)$ is the connected component of $\SING$ containing $x$. For each $z\in\omega(\mathbf{x})$, there exists $\sigma$ such that the sequence $z_{\sigma}$ converges to $z$. 


\begin{The}\label{asymptotic}
	Let $H$ be a mechanical Hamiltonian as in \eqref{eq:mech_intro} and $u_c$ be a weak KAM solution of \eqref{intr_HJ}. Suppose $x\in\SING$, $\mathbf{x}$ is the unique generalized characteristic starting from $x$, and $C(x)$ is the connected component of $\SING$ containing $x$.
	\begin{enumerate}[\rm (a)]
	\item If $\lim_{t\to\infty}\mathbf{x}(t)$ exists, then there exists $z\in\overline{C(x)}$ such that $0\in H^c_p(z,D^+u_c(z))$.
	\item Suppose that $\lim_{t\to\infty}\mathbf{x}(t)$ does not exist and fix any $\tau\in(0,t_0]$. Then for any $z_{\sigma}\in\mathcal{Z}^{\tau}$ such that $N_{\sigma}<\infty$, there exists a closed generalized characteristic contained in $\overline{\{\mathbf{x}(t):t>0\}}\subset\overline{C(x)}$. 
	\item Let $\tau_k\to 0^+$ as $k\to\infty$. If for each $k\in\N$, there exists an $\sigma_k$ such that $z^{\tau_k}_{\sigma_k}\in\mathcal{Z}^{\tau_k}$ with $N_{\sigma_k}^{\tau_k}<\infty$, and $\lim_{k\to\infty}\tau_kN_{\sigma_k}^{\tau_k}=0$, then there exists $z\in\overline{C(x)}$ such that $0\in \textup{co}\,H^c_p(z,D^+u_c(z))$.
	\item Fix any $\tau\in(0,t_0]$ and $z_{\sigma}\in\mathcal{Z}^{\tau}$. If $N_{\sigma}(\tau)=\infty$, then there exists a global generalized characteristic $\mathbf{y}:(-\infty,+\infty)\to\T^n$ such that $\{\mathbf{y}(t): t\in\R\}$ is contained in $\overline{\{\mathbf{x}(t):t>0\}}\subset\overline{C(x)}$.
	\end{enumerate}
\end{The}

\begin{Rem}
	In fact, all  results in the theorem above hold true if we only suppose $H$ has the uniqueness property. In the case that $H$ is a mechanical Hamiltonian, the convex hull in (c) will disappear. Unfortunately, it is still unknown whether there is another example of Hamiltonian which still has the uniqueness property.
\end{Rem}

\begin{proof}
	Fix any $x\in\SING$. Choose a sequence $\tau_n\to 0^+$ and let $z^{n}_i=\mathbf{x}(i\tau_n)$. Then $\lim_{i\to\infty}z^n_i=z\in\overline{C(x)}$. We claim that
	$$
	D_yA_{\tau_n}(z,z)\in D^+u_c(z).
	$$ 
	Indeed, let $\xi^n_i\in\Gamma^{\tau_n}_{z^n_i,z^n_{i+1}}$ be the minimizer of $A_{\tau_n}(z^n_i,z^n_{i+1})$. Then $\{\xi^n_i\}_{i\in\N}$ is equi-Lipschitz. So, by taking a subsequence, we obtain a Lipschitz curve $\xi\in\Gamma^{\tau_n}_{z,z}$ which is also a minimizer for $A_{\tau_n}(z,z)$. On the other hand, we have  $L^c_v(\xi^n_i(\tau_n),\dot{\xi}^n_i(\tau_n))\in D^+u_c(z^n_{i+1})$. Therefore,
	\begin{align*}
		D_yA_{\tau_n}(z,z) = L^c_v(\xi^n(\tau_n),\dot{\xi}^n(\tau_n)) \in D^+u_c(z)
	\end{align*}
	by the upper semicontinuity of the set valued map $x\rightrightarrows D^+u_c(x)$.
	
	We will show that $0\in H^c_p(z,D^+u_c(z))$. Indeed, we have
	\begin{equation}\label{eq:limit_velociy}
	\begin{split}
		|\dot{\xi}^n(\tau_n)|=\left|\frac{\xi^n(\tau_n)-z}{\tau_n}-\dot{\xi}^n(\tau_n)\right|
		&\leqslant\frac 1{\tau_n}\int^{\tau_n}_{0}|\dot{\xi}^n(s)-\dot{\xi}^n(\tau_n)|\ ds\\
		&\leqslant\frac {C_0}{\tau_n}\int^{\tau_n}_{0}(\tau_n-s)\ ds=\frac {C_0}2\tau_n.
	\end{split}
	\end{equation}
Therefore $\lim_{n\to\infty}\dot{\xi}^n(\tau_n)=0$. Recall that
$$
\dot{\xi}^n(\tau_n)=H^c_p(z,L^c_v(z,\dot{\xi}^n(\tau_n)))\in H^c_p(z,D^+u_c(z)).
$$
It follows $0\in H^c_p(z,D^+u_c(z))$. This proves (a).

Now let us prove part (b). Let $\sigma=\{i_1,i_2,\dots,i_k,\dots\}$. Since $\lim_{t\to\infty}\mathbf{x}(t)$ does not exist and $N^\tau_{\sigma}<\infty$, by the principle of pigeon hole, there exist a positive integer $N\geqslant 2$, and a sequence of pairs $(i_{k_j},i_{k_j+1})$ such that $i_{k_j+1}-i_{k_j}\equiv N$ for all $j\in\N$. Thus, without loss of generality, we can assume $i_{k+1}-i_{k}\equiv N$ for all $k\in\N$.

We denote by $\tilde{\mathbf{y}}^{\tau}_k:[i_k\tau,i_{k+1}\tau]\to\T^n$ a segment of $\mathbf{x}_{\tau}$ and define $\mathbf{y}^{\tau}_k:[0,N\tau]\to M$ as $\mathbf{y}^{\tau}_k(s)=\tilde{\mathbf{y}}^{\tau}(s+i_k\tau)$. The sequence $\{\mathbf{y}^{\tau}_k\}_{k\in\N}$ is equi-Lipschitz, thus it converges to a Lipschitz curve $\mathbf{y}^{\tau}:[0,N\tau]\to\T^n$ without loss of generality, and $\mathbf{y}^{\tau}(0)=\mathbf{y}^{\tau}(N\tau)=z^{\tau}$. Therefore $\mathbf{y}^{\tau}$ is a closed generalized characteristic by Lemma \ref{limiting_GC}. This completes the proof of (b).

The proof of (c) is a refinement of that of (b). Suppose $\tau_k\to 0^+$ as $k\to\infty$. If for each $k\in\N$, there exists $\sigma_k$ such that $z^{\tau_k}_{\sigma_k}\in\mathcal{Z}^{\tau_k}$ with $N_{\sigma_k}^{\tau_k}<\infty$, $\lim_{k\to\infty}\tau_kN_{\sigma_k}^{\tau_k}=0$. Let $N(k)\geqslant 2$ be as in the proof of (b) and $\mathbf{y}^{\tau_k}:[0,N(k)\tau_k]\to\T^n$ be the closed generalized characteristic with $\mathbf{y}^{\tau_k}(0)=\mathbf{y}^{\tau_k}(N(k)\tau_k)=z^{\tau_k}$ by (b). It is clear that $\lim_{k\to\infty}N(k)\tau_k=0$ since $N(k)\leqslant N_{\sigma_k}^{\tau_k}$. 

Let $\xi^k_j\in\Gamma^{\tau_k}_{\mathbf{y}^{\tau_k}(j\tau_k),\mathbf{y}^{\tau_k}((j+1)\tau_k)}$ be the minimizer for $A_{\tau_k}(\mathbf{y}^{\tau_k}(j\tau_k),\mathbf{y}^{\tau_k}((j+1)\tau_k))$, $j=0,\ldots,N(k)-1$. Thus
	\begin{equation}\label{eq:limit_velociy2}
	\begin{split}
		\left|\frac 1{N(k)}\sum_{j=0}^{N(k)-1}\dot{\xi}^k_j(\tau_k)\right|=&\left|\frac 1{N(k)}\sum_{j=0}^{N(k)-1}\frac{\xi^k_j(\tau_k)-\xi^k_j(0)}{\tau_k}-\frac 1{N(k)}\sum_{j=0}^{N(k)-1}\dot{\xi}^k_j(\tau_k)\right|\\
		\leqslant&\frac 1{N(k)}\sum_{j=0}^{N(k)-1}\frac 1{\tau_k}\int^{\tau_k}_0|\dot{\xi}^k_j(s)-\dot{\xi}^k_j(\tau_k)|\ ds\leqslant\frac{C_0}{2}\tau_k.
	\end{split}
	\end{equation}
	Due to nonsmooth Fermat rule, we have that 
	\begin{align*}
		L^c_v(\xi^k_j(\tau_k),\dot{\xi}^k_j(\tau_k))=D_yA_{\tau_k}(\mathbf{y}^{\tau_k}(j\tau_k),\mathbf{y}^{\tau_k}((j+1)\tau_k))\in D^+u_c(\xi^k_j(\tau_k)),
	\end{align*}
	or $\dot{\xi}^k_j(\tau_k)\in H^c_p(\xi^k_j(\tau_k),D^+u_c(\xi^k_j(\tau_k)))$. Therefore
	$$
	\frac 1{N(k)}\sum_{j=0}^{N(k)-1}\dot{\xi}^k_j(\tau_k)\in\frac 1{N(k)}\sum_{j=0}^{N(k)-1}H^c_p(\xi^k_j(\tau_k),D^+u_c(\xi^k_j(\tau_k))).
	$$
	By \eqref{eq:limit_velociy2}, the condition $\lim_{k\to\infty}N(k)\tau_k=0$, and the compactness, there exists $z\in\overline{C(x)}$ such that
	$$
	0\in\text{co}\,H^c_p(z,D^+u_c(z)).
	$$
	This proves (c).
	
	Finally, we turn to the proof of (d). Let $\sigma=\{i_1,i_2,\ldots,i_k,\ldots\}$. Since there exists a sequence of pairs $(i_{k_j},i_{k_j+1})$ such that $\lim_{j\to\infty}(i_{k_j+1}-i_{k_j})=\infty$ by the assumption $N_{\sigma}^{\tau}=\infty$, we assume the sequence $\{i_{k+1}-i_k\}$ to be strictly increasing without loss of generality.
	
	Now, set $z^{\tau}_k=\mathbf{x}(i_k\tau)$ ($k\in\N$). Then $\{z^{\tau}_k\}$ converges to some $z^{\tau}\in\overline{C(x)}$ as $k\to\infty$. For any positive integer $k$, define $\mathbf{x}^k_{\tau}:[0,(i_{k+1}-i_k)\tau]\to\T^n$ by $\mathbf{x}^k_{\tau}(s)=\mathbf{x}_{\tau}(s+i_{k}\tau)$, then $\mathbf{x}^k_{\tau}(0)=z^{\tau}_k$. Invoking Lemma \ref{limiting_GC}, there exists $\mathbf{z}:[0,+\infty)\to\T^n$ such that $\mathbf{x}^k_{\tau}$ converges uniformly to $\mathbf{z}$ on each compact subinterval of $[0,+\infty)$. Moreover, $\mathbf{z}$ satisfies the generalized characteristic inclusion. On the other hand, we define $\mathbf{y}^k_{\tau}:[-(i_{k+1}-i_k)\tau,0]\to\T^n$ by $\mathbf{y}^k_{\tau}(s)=\mathbf{x}_{\tau}(s+i_{k+1}\tau)$. Similarly, by Ascoli-Arzela theorem and the upper-semicontinuity of the compact-valued multifuncntion $x\rightrightarrows\text{co}\,H^c_p(x,D^+u_c(x))$, there exists $\mathbf{y}:(-\infty,0]\to\T^n$ such that $\mathbf{y}^k_{\tau}$ converges uniformly to $\mathbf{y}$ on each compact subinterval of $(-\infty,0]$. The juxtaposition of $\mathbf{y}$ and $\mathbf{z}$ gives the desired global generalized characteristic. 
\end{proof}

We need reformulate Theorem \ref{asymptotic}.

\begin{The}
	Suppose $H$ is a mechanical Hamiltonian as in \eqref{eq:mech_intro}, $u_c$ is a weak KAM solution of \eqref{intr_HJ}, $x\in\SING$, and $C(x)$ is the connected component of $\SING$ containing $x$. Let $\mathbf{x}:[0,+\infty)\to\T^n$ be the unique generalized characteristic staring from $x$. If there is no critical point of $u_c$ with respect to $H^c$ in $\overline{C(x)}$, then $\lim_{t\to\infty}\mathbf{x}(t)$ does not exists. In addition, for any $\sigma\in \mathcal{Z}^{t_0}$,  the following properties hold:
	\begin{enumerate}[\rm (a)]
	\item If $ N_{\sigma}<\infty$, then there exists a non-constant closed generalized characteristic $\mathbf{y}:[0,T]\to\T^n$, $\mathbf{y}(0)=\mathbf{y}(T)$, contained in $\overline{\{\mathbf{x}(t):t>0\}}\subset\overline{C(x)}$. Moreover, we have either
	\begin{enumerate}[\rm (1)]
	\item $\mathbf{y}:[0,T]\to\T^n$ is a $C^2$ closed regular characteristic contained in $\mathcal{I}_{u_c}$, or
	\item $\mathbf{y}:[0,T]\to\T^n$ is a closed singular generalized characteristic.
	\end{enumerate}
	\item If $N_{\sigma}=\infty$\footnote{We can always suppose that $N_{\sigma}=\infty$ by choosing a suitable subsequence of $\sigma$ if necessary.}, then there exists a global generalized characteristic $\mathbf{y}:\R\to\T^n$ such that $\{\mathbf{y}(t): t\in\R\}$ is contained in $\overline{\{\mathbf{x}(t):t>0\}}\subset\overline{C(x)}$. Moreover, we have either
	\begin{enumerate}[\rm (1)]
	\item $\mathbf{y}:\R\to\T^n$ is a global singular generalized characteristic, or
	\item $\overline{C(x)}\cap\mathcal{I}_{u_c}\neq\emptyset$. In particular, $\overline{C(x)}\cap\mathcal{A}(H^c)\neq\emptyset$.

	\end{enumerate}
	\end{enumerate}
\end{The}

\begin{proof}
	The first statement is a direct consequence of Theorem \ref{asymptotic}. 
	
	We now turn to the proof of (a). If $N_{\sigma}<\infty$, in view of Theorem \ref{asymptotic} (b), there exists a non-constant closed generalized characteristic $\mathbf{y}:[0,Nt_0]\to\T^n$, $\mathbf{y}(0)=\mathbf{y}(Nt_0)$. For $j=0,1,\ldots,N$, set $z_j=\mathbf{y}(jt_0)$. Suppose that $z_0=z_N$ is a point of differentiability (resp. singular point) of $u_c$, then $\mathbf{y}(s)$ with $s\in[(N-1)t_0,Nt_0]$ (resp. $s\in[0,t_0]$), are all points of differentiability (resp. singular points) of $u_c$ by our uniqueness assumption on the generalized characteristics and \cite[Lemma 3.2]{Cannarsa-Cheng3}. Therefore, by induction, we can prove that $\mathbf{y}(s)$, $s\in[0,Nt_0]$, are all points of differentiability (resp. singular points) of $u_c$. Moreover, for the case that $z_0=z_N$ is a differentiable point, $\{\mathbf{y}(s): s\in[0,Nt_0]\}$ is contained in $\mathcal{I}_{u_c}$. This completes the proof of part (a).
		
	The first part of (b) is already known. As for the second part, we know there exists a global generalized characteristic $\mathbf{y}:(-\infty,+\infty)\to\T^n$ contained in $\overline{C(x)}$. If $\mathbf{y}$ is not a global singular generalized characteristic, then there exists a point of differentiability, say $\mathbf{y}(0)$, of $u_c$. Then $\mathbf{y}(s)$, $s\leqslant0$, are all points of differentiability of $u_c$. Since the uniqueness assumption, it is clear that $\mathbf{y}:(-\infty,0]\to\T^n$ is a $u_c$-calibrated curve whose $\alpha$-limit set is contained in the Aubry set $\mathcal{A}(H^c)$, see \cite{Fathi-book}. Therefore $\overline{C(x)}\cap\mathcal{A}(H^c)\neq\varnothing$. 
\end{proof}

\begin{The}
	Suppose $H$ is a mechanical Hamiltonian as in \eqref{eq:mech_intro}, $x\in\SING$ and $\mathbf{x}$ is the unique generalized characteristic starting from $x$. If there exists a point of differentiability of $u_c$ in $\omega(\mathbf{x})$, then $\omega(\mathbf{x})\cap\mathcal{A}(H^c)\not=\varnothing$.
\end{The}

\begin{proof}	
	By assumption, there exists a $\sigma=\{i_1,i_2,\ldots,i_k,\ldots\}$ such that $\lim_{k\to\infty}\mathbf{x}(i_kt_0)=z$ where $u_c$ is differentiable at $z$, $\lim_{j\to\infty}(i_{k_j+1}-i_{k_j})=\infty$ and the sequence $\{i_{k+1}-i_k\}$ is strictly increasing without loss of generality.
	
	Now let $z_k=\mathbf{x}(i_kt_0)$ ($k\in\N$). For any positive integer $k$, define $\mathbf{x}^k:[0,(i_{k+1}-i_k)t_0]\to\T^n$ by $\mathbf{x}^k(s)=\mathbf{x}(s+i_{k}t_0)$. In particular, $\mathbf{x}^k(0)=z_k$. Invoking Lemma \ref{limiting_GC}, there exists $\mathbf{x}^+:[0,+\infty)\to\T^n$ such that $\mathbf{x}^k$ converges uniformly to $\mathbf{x}^+$ on any compact subinterval of $[0,+\infty)$. Moreover, $\mathbf{x}^+$ satisfies the generalized characteristic inclusion. On the other hand, we define $\mathbf{y}^k:[-(i_{k+1}-i_k)t_0,0]\to\T^n$ by $\mathbf{y}^k(s)=\mathbf{x}(s+i_{k+1}t_0)$. Similarly, by the Ascoli-Arzela theorem and the upper-semicontinuity of the compact-valued multifuncntion $z\rightrightarrows\text{co}\,H^c_p(x,D^+u_c(x))$, there exists $\mathbf{x}^-:(-\infty,0]\to\T^n$ such that $\mathbf{y}^k$ converges uniformly to $\mathbf{x}^-$ on any compact subinterval of $(-\infty,0]$. The juxtaposition of $\mathbf{x}^-$ and $\mathbf{x}^+$ gives the desired global generalized characteristic which is denoted by $\mathbf{z}$. Notice that $\overline{\{\mathbf{z}(s):s\in\R\}}\subset\omega(\mathbf{x})$ and $\mathbf{z}(0)=z$.
	
	Since $u_c$ is differentiable at $z$, we claim that $\mathbf{z}(s)$, $s\leqslant0$, are all points of differentiability of $u_c$. Indeed, for any $s\in[-t_0,0)$, by the uniqueness of the generalized characteristics, we have that $\mathbf{z}(0)$ is the unique maximizer of the function $u_c(\cdot)-A_{-s}(\mathbf{z}(s),\cdot)$. Thus, 
	$$
	L^c_v(\xi(0),\dot{\xi}(0))=D_yA_{-s}(\mathbf{z}(s),\mathbf{z}(0))=Du_c(\mathbf{z}(0)),
	$$
	where $\xi:[s,0]\to\T^n$ is a minimizer for $A_{-s}(\mathbf{z}(s),\mathbf{z}(0))$. Since $u_c$ is differentiable at $\mathbf{z}(0)$, there exists a unique backward calibrated curve $\gamma$ on $(-\infty,0]$ such that $\gamma(0)=\mathbf{z}(0)$ and $L^c_v(\gamma(0),\dot{\gamma}(0))=Du_c(\mathbf{z}(0))$. It follows that $\xi$ and $\gamma$ coincide on $[s,0]$ and $u_c$ is differentiable at $\mathbf{z}(s)$.
	
	We also claim that $\mathbf{z}:(-\infty,0]\to\T^n$ is a $u_c$-calibrated curve. For convenience, we work on $\R^n$. Indeed, suppose $s\leqslant0$, by the formula of the directional derivative of semiconcave function, we have that
\begin{align*}
	\frac d{ds}u_c(\mathbf{z}(s))=&\,Du_c(\mathbf{z}(s))\cdot\dot{\mathbf{z}}(s)= L^c_v(\mathbf{z}(s),\dot{\mathbf{z}}(s)) \cdot\dot{\mathbf{z}}(s)\\
	=&\,H^c(\mathbf{z}(s),L^c_v(\mathbf{z}(s),\dot{\mathbf{z}}(s)))+L^c(\mathbf{z}(s),\dot{\mathbf{z}}(s))\\
	=&\,H^c(\mathbf{z}(s),Du_c(\mathbf{z}(s)))+L^c(\mathbf{z}(s),\dot{\mathbf{z}}(s))=L^c(\mathbf{z}(s),\dot{\mathbf{z}}(s)).
	\end{align*}
Then, by integrating the equality above for any $s_1<s_2\leqslant 0$, we conclude that $\mathbf{z}$ is a $u_c$-calibrated curve.
	
	It is clear that $\overline{\{\mathbf{z}(s):s\in\R\}}\subset\omega(\mathbf{x})$ and $\overline{\{\mathbf{z}(s):s\in\R\}}\cap\mathcal{I}_{u_c}\not=\varnothing$ since the $\alpha$-limit set of $\mathbf{z}$ is contained in $\mathcal{A}(H^c)$. Therefore, we conclude that $\omega(\mathbf{x})\cap\mathcal{A}(H^c)\not=\varnothing$.
\end{proof}

\section{An example where singularities approach the Aubry set}

Here we give an example based on Bangert's work on minimal geodesics on 2-torus $\T^2$, although it is well known as Aubry-Mather theory of area-preserving monotone twist maps. The standard literatures include Bangert's survey article \cite{BangertBook} and Mather-Forni's lecture notes \cite{Mather-Forni}.

Let $\R^{\Z}$ be the set of bi-infinite sequences of real numbers with the product topology, and an element in $\R^{\Z}$ will be denoted by $\{x_i\}_{i\in\Z}$. Given a function $h:\R^2\to\R$, which is called a {\em generating function of a variational principle}, a segment $(x_j,\ldots,x_k)$, $j<k$, is called {\em minimal} if
$$
h(x_j,x_{j+1},\ldots,x_k):=\sum_{i=j}^{k-1}h(x_i,x_{i+1})\leqslant\sum_{i=j}^{k-1}h(y_i,y_{i+1})
$$
for all $(y_j,\ldots,y_k)$ with $y_j=x_j$ and $y_k=x_k$. A bi-infinite sequence $\{x_i\}$ is said to be {\em minimal} if every finite segment of $\{x_i\}$ is minimal. We always call a minimal bi-infinite sequence $\{x_i\}\in\R^{\Z}$ a {\em minimal configuration}, and  denote by $\mathcal{M}=\mathcal{M}(h)$ the set of minimal configurations with respect to the generating function $h$. To define a monotone twist map by a generating function $h$, we suppose $h$ satisfies the following conditions:
\begin{enumerate}[(h1)]
  \item $h(x+1,y+1)=h(x,y)$ for all $(x,y)\in\R^2$;
  \item $\lim_{|y|\to+\infty}h(x,x+y)=+\infty$ uniformly in $x$;
  \item If $x_1<x_2$ and $y_1<y_2$, then
  $$
  h(x_1,y_1)+h(x_2,y_2)<h(x_1,y_2)+h(x_2,y_1);
  $$
  \item If $(x_{-1},x_0,x_1)\not=(y_{-1},y_0,y_1)$ are minimal and $x_0=y_0$, then $(x_{-1}-y_{-1})(x_1-y_1)<0$.
\end{enumerate}
When $h$ is smooth, it is to define an area-preserving monotone twist map $F:\mathbb{S}^1\times\R\to\mathbb{S}^1\times\R$ such that its lift $\tilde{F}:\R\times\R\to\R\times\R$ satisfying
$$
\tilde{F}(x_0,y_0)=(x_1,y_1)\quad\Leftrightarrow\quad y_0=-D_1h(x_0,x_1),\ y_1=D_2h(x_0,x_1).
$$

There exists a continuos map $\tilde{\rho}:\mathcal{M}\to\R$ such that, if $\xi=\{x_i\}\in\mathcal{M}$ then $|x_i-x_0-i\tilde{\rho}(x)|<1$ for all $i\in\Z$. In fact, $\tilde{\rho}(\xi)=\lim_{i\to\infty}(x_i-x_0)/i$ and is called the {\em rotation number} of the minimal configuration $\xi$. Moreover, by the order-preserving properties of the minimal configurations (Lemma of Aubry graphs), one could  associate a minimal configuration $\xi=\{x_i\}$ with a (order-preserving) circle map $f$ satisfying $f(x_i)=x_{i+1} \pmod 1$, $i\in\Z$, and the Poincar\'e's rotation number $\rho(f)=\tilde{\rho}(\xi)$. If $\rho(f)=\omega$ is irrational, then the recurrent set $\mathcal{M}^{rec}(f)$ of $f$ is either the whole circle  or a Cantor subset which is usually called Denjoy set.

Let $g$ be a Riemannian metric on $\T^2$ with $d$ its Riemannian distance and $\tilde{d}$ the $\Z^2$-periodic distance which is the lift of $d$ to $\R^2$. An associated generating function $h$ is defined as followed: Let $(i,x_i)\in\Z\times\R^1$, define
\begin{equation}\label{eq:generating_function_1}
	h(x_i,x_{i+1})=\tilde{d}((i,x_i),(i+1,x_{i+1})).
\end{equation}
It is known that such $h$ is a generating function satisfying the condition (h1)-(h4). If  $\gamma:\R\to\T^2$ is a {\em minimal geodesic\footnote{Due to Bangert, a geodesic is said to be minimal if for any $[a,b]$, $d(\gamma(a),\gamma(b))=\text{length}\,(\gamma|_{[a,b]})$.}}, then $\{\tilde{\gamma}(i)\}$ is a minimal configuration where $\tilde{\gamma}$ is a lift of $\gamma$ to $\R^2$ (see Figure \ref{fig_1}). 

\begin{figure}
	\begin{center}
\begin{tikzpicture}[scale=1]
\draw[-] [name path=hline 1] (-0.5,0) -- (5.5,0);
\draw[-] [name path=vline 1] (0,-0.5) -- (0,2.5);
\draw[-] [name path=vline 2] (1,0) -- (1,2.5);
\draw[-] [name path=vline 3] (2,0) -- (2,2.5);
\draw[-] [name path=vline 4] (3,0) -- (3,2.5);
\draw[-] [name path=vline 5] (4,0) -- (4,2.5);
\draw[-] [name path=vline 6] (5,0) -- (5,2.5);
\draw [red][name path=curve 1] (-0.5,0.3) .. controls (2,0.8) and (4,2.0) .. (5.3,2.3);
\fill [name intersections={of=vline 1 and curve 1, by={a1}}] 
        (a1) circle (1.5pt); 
        \draw (-0.1,0.3) node[right] {$\scriptscriptstyle (0,x_0)$};
\fill [name intersections={of=vline 2 and curve 1, by={a2}}] 
        (a2) circle (1.5pt);
        \draw (0.9,0.5) node[right] {$\scriptscriptstyle (1,x_1)$};
\fill [name intersections={of=vline 3 and curve 1, by={a3}}] 
        (a3) circle (1.5pt);
        \draw (1.9,0.85) node[right] {$\scriptscriptstyle (2,x_2)$};
\fill [name intersections={of=vline 4 and curve 1, by={a4}}] 
        (a4) circle (1.5pt);
        \draw (2.9,1.25) node[right] {$\scriptscriptstyle (3,x_3)$};
\fill [name intersections={of=vline 5 and curve 1, by={a5}}] 
        (a5) circle (1.5pt);
      \draw (3.9,1.65) node[right] {$\scriptscriptstyle (4,x_4)$};
\fill [name intersections={of=vline 6 and curve 1, by={a6}}] 
        (a6) circle (1.5pt);
      \draw (4.9,2.05) node[right] {$\scriptscriptstyle (5,x_5)$};
\end{tikzpicture}
\end{center}
\caption{The red line is the minimal configuration of a minimal geodesic.}\label{fig_1}
\end{figure}
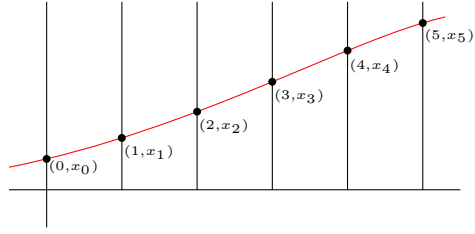

Now we consider a  {\em big bump \footnote{A big bump on a torus $(\T^2,g)$ is a closed disk $B\subset\T^2$ with the following property: there exists  $x\in B$ so that the (Riemannian) distance from $x$ to  $\partial B$ is larger than $\frac 14\text{length}\, (\partial B)$.}} Riemannian metric $g_x(v,v)=\frac{1}{2}\langle A(x)v,v\rangle$ (See \cite{BangertBook}). As a special example of what in \eqref{eq:mech_intro}, we consider the corresponding  Hamiltonian   
\begin{equation}\label{eq:geodesic_Hamiltonian}
	H(x,p)=\frac{1}{2}\langle A^{-1}(x)p,p\rangle,\quad (x,p)\in\T^2\times\R^2.
\end{equation}

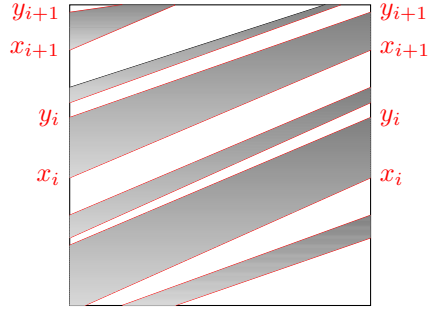
\begin{figure}
	\begin{center}
\begin{tikzpicture}[scale=1]
\draw[-] (0,0) -- (4,0) -- (4,4) -- (0,4) -- (0,0);
\draw [red] (0,2.5) node[left] {$y_i$} -- (4,3.9) node[right] {$y_{i+1}$};
\draw [red] (0,1.7) node[left] {$x_i$} -- (4,3.4) node[right] {$x_{i+1}$};
\shade[top color=gray,bottom color=gray!30] (0,1.7) -- (0,2.5) -- (4,3.9) -- (4,3.4) -- (0,1.7);
\draw [red] (0,3.9) node[left] {$y_{i+1}$} -- (0.7,4);
\draw [red] (0,3.4) node[left] {$x_{i+1}$} -- (1.4,4);
\shade[top color=gray,bottom color=gray!30] 
     (0,3.4) -- (0,3.9) -- (0.7,4) -- (1.4,4) -- (0,3.4);
\draw [red] (0.7,0) -- (4,1.2);
\draw [red] (1.4,0) -- (4,0.9);
\shade[top color=gray,bottom color=gray!30] 
     (0.7,0) -- (4,1.2) -- (4,0.9) -- (1.4,0) -- (0.7,0);
\draw [red] (0,0.9) -- (4,2.7);
\draw [red] (0,1.2) -- (4,2.9);
\shade[top color=gray,bottom color=gray!30] 
(0,1.2) -- (4,2.9) -- (4,2.7) -- (0,0.9) -- (0,1.2);
\draw [red] (0,2.7) -- (3.6,4);
\draw (0,2.9) -- (3.4,4);
\shade[top color=gray,bottom color=gray!30]
(0,2.7) -- (3.6,4) -- (3.4,4) -- (0,2.9) -- (0,2.7);
\draw [red] (0,0.8) -- (4,2.5);
\draw [red] (0.2,0) -- (4,1.7);
\shade[top color=gray,bottom color=gray!30]
(0,0) -- (0.2,0) -- (4,1.7) -- (4,2.5) -- (0,0.8) -- (0,0);
\draw [red] (4,2.5) node[right] {$y_{i}$};
\draw [red] (4,1.7) node[right] {$x_{i}$};
\end{tikzpicture}
\end{center}
\caption{The shaded part stands for the complement of the Denjoy set. The red line is a minimal geodesic with rotation number $\alpha\in\R\setminus\Q$.}\label{fig_2}
\end{figure}

Let $\mathcal{M}_\omega\subset\T^2$ denote the set of all minimal geodesics with rotation number $\omega$, denote by  $\mathcal{M}^{rec}_\omega$ the set of recurrent points in $\mathcal{M}_\omega$. Since $g$ is  big bump type, then for any  
$\omega\in\R\setminus\Q$, $\mathcal{M}^{rec}_\omega$ is a Denjoy set(\cite[Proposition 9.7]{BangertBook}).  On the other hand, it is well-known in Aubry-Mather theory and weak KAM theory that there exists some $c\in\R$ such that the projected Aubry set $\mathcal{A}(H^c)=\mathcal{M}^{rec}_\omega$.   

Restricted on the vertical line $\{0\}\times\T$, the complement of the Denjoy set $\mathcal{M}^{rec}_\omega$ is composed of a countable family of open intervals (see Figure \ref{fig_2}). Let $\{(x_i,y_i)\}$ be a sequence of intervals produced by the iteration of any of such an open interval, say $(x_0,y_0)$, then these intervals can not intersect each other by the property of order-preserving. This leads to $\sum_{i\in\Z}(y_i-x_i)\leqslant1$ by the periodicity. Therefore, for any $\varepsilon>0$, there exists $i_0\in\N$ such that $\sum_{i>i_0}(y_i-x_i)<\varepsilon$. Let $u_c$ be a weak KAM solution of the Hamilton-Jacobi equation
\begin{equation}\label{example_HJE}
	H(x,c+Du(x))=\alpha_H(c).
\end{equation}
If $\mathbf{x}$ and $\mathbf{y}$ are the unique generalized characteristics with respect to  $u_c$, starting from $x$ and $y$ respectively, then there exists $C>0$ such that
	\begin{equation}\label{eq:depedence_G_C}
			|\mathbf{x}(s)-\mathbf{y}(s)|\leqslant C|x-y|, \quad s\in[0,1],
	\end{equation}
	since Proposition \ref{properties_g}. 
	
Now we claim that arbitrary $\epsilon$-neighborhood of $\mathcal{A}(H^c)=\mathcal{M}^{rec}_\omega$ contains singular points of $u_c$. Indeed,  for each $a\in(x_i,y_i)$ with some fixed $i>i_0$, define
$$
T_a=\sup\{t: \mathbf{x}(t,0,(0,a))\not\in\text{Sing}\, (u_c)\},
$$
then $T_a<\infty$ by Proposition \ref{properties_g} (a). Denote by $n_a=[T_a]$ the integer part of $T_a$ and set $(0,a')=\mathbf{x}(n_a+1,0,(0,a))$. Then $a'\in(x_{i+n_a+1},y_{i+n_a+1})$, hence our claim is direct from the definition of $T_{a'}$ and $\sum_{i>i_0}(y_i-x_i)<\varepsilon$ and \eqref{eq:depedence_G_C}. 
 
In conclusion, for the Hamiltonian in \eqref{eq:geodesic_Hamiltonian}, we have
\begin{Pro}
Let $\omega\in\R\setminus\Q$ and $c\in\R^2$ be described as above, then for any weak KAM solution $u_c$ of \eqref{example_HJE}, the singular set $\mbox{\rm Sing}\, (u_c)$ intersects any neighborhood of the projected Aubry set $\mathcal{A}(H^c)$.	
\end{Pro}

\appendix
\section{Torus case}\label{se:torus}
In this section, we adapt our results in \cite{Cannarsa-Cheng3} to the flat $n$-torus $\T^n$. Given a Lagrangian $L$  on $\T^n\times\R^n$, of class $ C^2$, we can lift $L$ to the universal covering space $\R^n\times\R^n$ and denote the lifted Lagrangian by $L$ as well. Then $L(x,v)$ is $\T^n$-periodic in $x$. For our regularity results, we suppose that $L$ satisfies the following conditions (L1) and (L2):
\begin{enumerate}[\rm(L1)]
	\item \emph{Uniform convexity}  There exists a nonincreasing function $\nu:\R_+\to\R_+$ such that $L_{vv}(x,v)\geqslant \nu(|v|)I$, for all $(x,v)\in\R^n\times\R^n$.
	\item \emph{Superlinear growth}  There exist two superlinear and nondecreasing functions  $\theta, \Theta: \R_+\to\R_+$ and a constant $c_0>0$ such that 
	 $$\Theta(|v|)\geqslant L(x,v)\geqslant \theta(|v|)-c_0,\quad \textup{for all~} (x,v)\in\R^n\times\R^n$$
\end{enumerate}

Let $Q:=(0,1]^n$ be a fundamental domain of the flat $n$-torus $\T^n$ lifted to the universal covering space $\R^n$. For each $x,x'\in\R^n$, we say that $x\sim x'$ if $x-x'\in\Z^n$ and we denote by $[x]$ the equivalence class of $x$. For any $x,y\in\R^n$ and $t>0$, we denote by $A_t([x],[y])$ the fundamental solution for $L$ on the torus, which is defined  as follows:
\begin{equation}\label{eq:fund_sol_torus}
	A_t([x],[y])=\inf_{x\in[x],y\in[y]}A_t(x,y),\qquad \forall [x],[y]\in\T^n.
\end{equation}
Similarly, denote by $H$  the associated $\T^n$-periodic Hamiltonian. Let $u$ be a $\T^n$-periodic viscosity solution of the Hamilton-Jacobi equation
$$
H(x,Du(x))=\alpha_H(0), \qquad x\in\R^n.
$$
Then the weak KAM solution on the torus $\T^n$, associated with $u$, has the form
\begin{equation}
	u([x])=u(x)\quad\mbox{for any}~ x \in\R^n.
\end{equation}

\begin{Lem}\label{fund_sol_torus}
Let $L$ be a $\T^n$-periodic Tonelli Lagrangian satisfying {\rm (L1)} and {\rm (L2)}. There exists a small $t_0>0$ such that for any $0<t\leqslant t_0$, if $d([x],[y])<t$, there exist $x,y\in Q$  satisfying $x\in [x]$, $y\in [y]$, $|x-y|<t$ and $A_t([x],[y])=A_t(x,y)$.
\end{Lem}

\begin{proof}
For any $t>0$ and $x,y\in\R^n$, let $\sigma(s)=x+\frac st(y-x)$ for all $s\in[0,t]$. Then
\begin{equation}
\label{eq:torus}
A_t(x,y)\leqslant \int^t_0L(\sigma(s),\dot{\sigma}(s)) \leqslant t\kappa_1(|y-x|/t),
\end{equation}
where $\kappa_1(r)=\max_{z\in\T^n, |v|\leqslant r}L(z,v)$. On the other hand, for fixed $k\geqslant0$, we have
\begin{align}\nonumber
	A_t(x,y)=&\inf_{\xi\in\Gamma^t_{x,y}}\int^t_0L(\xi(s),\dot{\xi}(s))\ ds\geqslant\inf_{\xi\in\Gamma^t_{x,y}}\int^t_0\theta(|\dot{\xi}(s)|)-c_0\ ds\\\label{eq:torus2}
	\geqslant&\inf_{\xi\in\Gamma^t_{x,y}}\int^t_0\big(k|\dot{\xi}(s)|-\theta^*(k)-c_0\big)  ds\\\nonumber
	=& \,k|y-x|-t(\theta^*(k)+c_0)=C_1|y-x|-tC_2.
\end{align}
Let $x,y\in Q$ be such that $x\in [x]$, $y\in [y]$.  Setting 
$$
t_0=\frac{C_1}{\kappa_1(1)+C_2+C_1}<1,
$$
then for any $0<t\leqslant t_0$ and $|y-x|<t$, we have
$$
t\kappa_1(|y-x|/t)+tC_2+C_1|y-x|\leqslant C_1\leqslant C_1|y-y'|, \quad \textup{for all~} y'\notin Q \textup{~with~} y'\in[y].
$$
since $|y-y'|\geqslant 1$. By the above inequality, \eqref{eq:torus} and \eqref{eq:torus2}  we obtain
\begin{align*}
	A_t(x,y)\leqslant C_1(|y-y'|-|y-x|)-tC_2
	\leqslant C_1|y'-x|-tC_2\leqslant A_t(x,y').
\end{align*}
This leads to our conclusion.
\end{proof}

By appealing to Lemma \ref{fund_sol_torus} and the compactness of $\T^n$, one can adapt the proof of all the results in \cite{Cannarsa-Cheng3} and realize that these regularity properties of the fundamental solution hold in the torus case as well. Similarly, the global propagation result was obtained thanks to the local regularity properties and uniform estimates for fundamental solutions that are in turn consequences of our assumptions on the Lagrangian.
Since such estimates are valid for the torus in view of the compactness, global propagation holds as well.

%

\end{document}